\documentclass[11pt]{amsart}
\usepackage{geometry}                
\usepackage{color,graphicx}
\usepackage{amssymb}
\usepackage{epstopdf}
\usepackage{tabmac}
\usepackage{verbatim}

{\theoremstyle{plain}
\newtheorem{theorem}{Theorem}[section]

\newtheorem{proposition}[theorem]{Proposition}
\newtheorem{lemma}[theorem]{Lemma}
\newtheorem{corollary}[theorem]{Corollary}
}
{\theoremstyle{definition}
\newtheorem{definition}[theorem]{Definition}
\newtheorem*{definition*}{Definition}
\newtheorem*{proposition*}{Proposition}
\newtheorem{remark}[theorem]{Remark}
\newtheorem{example}[theorem]{Example}

}

\numberwithin{equation}{section}

\newcommand{\A}{\mathcal{A}}
\newcommand{\ascents}{\mathrm{asc}}
\newcommand{\bh}{\mathbf{h}}
\newcommand{\be}{\mathbf{e}}
\newcommand{\bs}{\mathbf{s}}
\newcommand{\sig}{\mathbf{\mathfrak{S}}}
\newcommand{\bp}{\mathbf{p}}

\newcommand{\C}{\mathcal{C}}

\newcommand{\core}{\mathrm{core}}
\newcommand{\HH}{\mathcal{H}}
\newcommand{\height}{\mathrm{ht}}

\newcommand{\la}{\lambda}
\newcommand{\len}{\mathrm{len}}

\newcommand{\supp}{\mathrm{supp}}

\newcommand{\Park}[1]{{{\mathcal P}^{(k)}_{#1}}}
\newcommand{\QQ}{\mathbb{Q}}
\newcommand{\ZZ}{\mathbb{Z}}

\newcommand{\pchoose}[2]{\left(\substack{#1\\ #2}\right)}

\newcommand{\uincA}{u_A^{\mathrm{inc}}}
\newcommand{\udecA}{u_A^{\mathrm{dec}}}

\title{The Murnaghan--Nakayama rule for $k$-Schur functions}

\author{Jason Bandlow}
\address{Department of Mathematics, University of Pennsylvania,
David Rittenhouse Laboratory, 209 South 33rd Street,
Philadelphia, PA 19104-6395, USA}
\email{jbandlow@math.upenn.edu}
\urladdr{http://www.math.upenn.edu/~jbandlow/}

\author{Anne Schilling}
\address{Department of Mathematics, One Shields Avenue, University of California, 
Davis, CA 95616, USA}
\email{anne@math.ucdavis.edu}
\urladdr{http://www.math.ucdavis.edu/~anne}
\thanks{Partially supported by the NSF grants DMS--0652641 and DMS--0652652.}

\author{Mike Zabrocki}
\address{York University, Mathematics and Statistics,
4700 Keele Street, Toronto, Ontario M3J 1P3, Canada}
\email{zabrocki@mathstat.yorku.ca}
\urladdr{http://garsia.math.yorku.ca/~zabrocki/}

\date{April 2010}

\begin{document}

\begin{abstract}
We prove the Murnaghan--Nakayama rule for $k$-Schur functions of Lapointe and
Morse, that is, we give an explicit formula for the expansion of the product of a power sum 
symmetric function and a $k$-Schur function in terms of $k$-Schur functions. This is proved 
using the noncommutative $k$-Schur functions in terms of the nilCoxeter algebra introduced
by Lam and the affine analogue of noncommutative symmetric functions of Fomin and Greene.
\end{abstract}

\maketitle

\section{Introduction}
The Murnaghan--Nakayama rule~\cite{LR:1934, Mur:1937, Nak:1941} is a combinatorial formula for
the characters $\chi_\la(\mu)$ of the symmetric group in terms of ribbon tableaux. Under the 
Frobenius characteristic map, there exists an analogous statement on the level of symmetric functions,
which follows directly from the formula
\begin{equation} \label{origMNrule}
	p_r s_\la = \sum_\mu (-1)^{\height(\mu/\la)} s_\mu.
\end{equation}
Here $p_r$ is the $r$-th power sum symmetric function, $s_\la$ is the Schur function
labeled by partition $\la$, and the sum is over all partitions $\la \subseteq \mu$ for which 
$\mu/\la$ is a border strip of size $r$. Recall that a border strip is a connected skew shape
without any $2\times 2$ squares. The height $\height(\mu/\la)$ of a border strip
$\mu/\la$ is one less than the number of rows.

In~\cite{FG:1998}, Fomin and Greene develop the theory of Schur functions in noncommuting
variables. In particular, they derive a noncommutative version of the Murnaghan--Nakayama
rule~\cite[Theorem 1.3]{FG:1998} for the nilCoxeter algebra (or more generally the local plactic 
algebra)
\begin{equation} \label{eq:mn_s}
	\bp_r  \bs_\la = \sum_w (-1)^{\ascents(w)} w \bs_\la \; ,
\end{equation}
where $w$ is a hook word of length $r$. Here $\bp_r$ and $\bs_\la$ are the noncommutative
analogues of the power sum symmetric function and the Schur function (introduced in
Section~\ref{s:notation}).  Consider $w$ as a word in a totally ordered alphabet that 
consists of the indices of the generators of the algebra.
The word $w$ is a hook word if $w = b_l b_{l-1} \ldots b_1 a_1 a_2 \ldots a_m$ where
\begin{equation}
\label{e:hook_word}
	b_l > b_{l-1} > \cdots  > b_1 > a_1 \leq a_2 \leq \cdots \leq a_m
\end{equation}
and $\ascents(w)=m-1$ is the number of ascents in $w$.
Actually, by~\cite[Theorem 5.1]{FG:1998} it can further be assumed that the
support of $w$ is an interval.

In this paper, we derive a (noncommutative) Murnaghan--Nakayama rule for the
$k$-Schur functions of Lapointe and Morse~\cite{LM:2007}.  $k$-Schur functions
form a basis for the ring $\Lambda_{(k)} = \ZZ[h_1,\ldots,h_k]$ spanned by the
first $k$ complete homogeneous symmetric functions $h_r$, which is a subring
of the ring of symmetric functions $\Lambda$.  Lapointe and
Morse~\cite{LM:2007} gave a formula for a homogeneous symmetric function $h_r$
times a $k$-Schur function (at $t=1$) as 
\begin{equation}
\label{e:kPieri}
	h_r s_\la^{(k)} = \sum_{\mu\in \Park{}} s_{\mu}^{(k)},
\end{equation}
where the sum is over all $k$-bounded partitions $\mu\in \Park{}$ such that
$\mu/\la$ is a horizontal $r$-strip and $\mu^{(k)}/\la^{(k)}$ is a vertical $r$-strip.
Here $\la^{(k)}$ denotes the $k$-conjugate of $\la$.
Equation~\eqref{e:kPieri} is a simple analogue of the Pieri rule for usual Schur functions, 
called the $k$-Pieri rule.  This formula can in fact be taken as the definition of $k$-Schur 
functions from which many of their properties can be derived.
Conjecturally, the $k$-Pieri definition of the $k$-Schur functions is equivalent to the
original definition by Lapointe, Lascoux, and Morse~\cite{LLM:2003} in terms of atoms.

Lam~\cite{Lam:2006} defined a noncommutative version of the $k$-Schur functions
in the affine nilCoxeter algebra as the dual of the affine Stanley symmetric functions
\begin{equation} \label{e:affineStanley}
	F_w(X) = \sum_{a=(a_1,\ldots,a_t)} \langle \bh_{a_t}(u) \bh_{a_{t-1}}(u) \cdots \bh_{a_1}(u) \cdot 1,
	w \rangle \; x_1^{a_1} \cdots x_t^{a_t},
\end{equation}
where the sum is over all compositions of $\len(w)$ satisfying $a_i \in [0,k]$. Here
\begin{equation*}
	\bh_r(u) = \sum_A \udecA
\end{equation*}
are the analogues of homogeneous symmetric functions in noncommutative variables,
where the sum is over all $r$-subsets $A$ of $[0,k]$ and $\udecA$ is the product
of the generators of the affine nilCoxeter algebra in cyclically decreasing order with
indices appearing in $A$. We denote the noncommutative analogue of $\Lambda_{(k)}$
by ${\bf \Lambda}_{(k)}$ as the subalgebra of the affine nilCoxeter algebra generated by
these analogues of homogeneous symmetric functions. See Section~\ref{ss:noncommutative} for further details.

Denote by $\bs_\la^{(k)}$ the noncommutative $k$-Schur function labeled by the $k$-bounded
partition $\la$ and $\bp_r$ the noncommutative power sum symmetric function in the affine 
nilCoxeter algebra. There is a natural bijection from $k$-bounded partitions
$\la$ to  $(k+1)$-cores, denoted $\core_{k+1}(\la)$ (see
Section~\ref{ss:partitions}).  We define a vertical domino in a skew-partition
to be a pair of cells in the diagram, with one sitting directly above the
other.  For the skew of two $k$-bounded partitions $\la \subseteq \mu$ we
define the height as
\begin{equation}
\label{e:height}
	\height(\mu/\la) = \text{number of vertical dominos in $\mu/\la$} \;.
\end{equation}
For ribbons, that is skew shapes without any $2\times 2$ squares,
the definition of height can be restated as the number of occupied rows minus the number
of connected components.  Notice that this is compatible with the usual definition of
the height of a border strip.

\begin{definition} \label{def:kribbon}
  The skew of two $k$-bounded partitions, $\mu / \lambda$, is called a
  \emph{$k$-ribbon of size $r$} if $\mu$ and $\lambda$ satisfy the following
  properties: 
  \begin{enumerate}
    \item[(0)] (containment condition) $\la \subseteq \mu$ and $\la^{(k)}
      \subseteq \mu^{(k)}$; \item[(1)] (size condition) $|\mu/\la|=r$;
    \item[(2)] (ribbon condition) $\core_{k+1}(\mu)/\core_{k+1}(\la)$ is a ribbon;
    \item[(3)] (connectedness condition) $\core_{k+1}(\mu)/\core_{k+1}(\la)$
      is $k$-connected (see Definition~\ref{definition:kconnected});
    \item[(4)] (height statistics condition) $\height(\mu / \la) + \height(
      \mu^{(k)} / \lambda^{(k)}) = r-1$.  
  \end{enumerate}
\end{definition}

Our main result is the following theorem.

\begin{theorem} \label{thm:main}
For $1\le r\le k$ and $\la$ a $k$-bounded partition, we have
\begin{equation*}
  \bp_r \bs_\lambda^{(k)} = \sum_\mu (-1)^{\height(\mu/\la)} \bs_\mu^{(k)},
\end{equation*}
where the sum is over all $k$-bounded partitions $\mu$ such that $\mu /
\lambda$ is a $k$-ribbon of size $r$.
\end{theorem}

If $k$ is sufficiently large, then 
$\mu/\la$ satisfies Definition~\ref{def:kribbon}
is equivalent to $\mu/\la$ is a connected ribbon of size
$r$ (as a skew partition). Hence for $k$ sufficiently large, 
Theorem~\ref{thm:main} implies Equation~\eqref{origMNrule}.

Let $\lambda, \nu$ be $k$-bounded partitions of the same size and $\ell$
the length of $\nu$. 
A $k$-ribbon tableau of shape $\lambda$ and type $\nu$
is a filling, $T$, of the cells of $\lambda$ with the labels $\{1,2,\ldots, \ell\}$
which satisfies the following conditions for all $1\le i\le \ell$:
\begin{enumerate}
  \item[(i)] the shape of the restriction of $T$ to the cells labeled
    $1,\ldots,i$ is a partition, and
  \item[(ii)] the skew shape $r_i$, which is the restriction of $T$ to the cells
    labeled $i$, is a $k$-ribbon of size $\nu_i$.
\end{enumerate}
We also define
\[ 
  \chi^{(k)}_{\lambda,\nu} = \sum_{T}^{} \left( \prod_{i=1}^{\ell} (-1)^{\height(r_i)} \right) \; ,
\]
where the sum is over all $k$-ribbon tableaux $T$ of shape $\lambda$ and
type $\nu$.

Iterating Theorem~\ref{thm:main} gives the following corollary. We remark that this formula
may also be considered as a definition of the $k$-Schur functions.
\begin{corollary} \label{cor:main}
  For $\nu$ a $k$-bounded partition, we have
  \[ \bp_\nu = \sum_{\lambda\in \Park{}}^{} \chi^{(k)}_{\lambda,\nu} \; \bs^{(k)}_\lambda. \]
\end{corollary}

All notation and definitions regarding our main Theorem~\ref{thm:main} are given in
Section~\ref{s:notation}. In Section~\ref{s:notation} we also see that there is a ring isomorphism
\[
  \iota : {\bf \Lambda}_{(k)} \to \Lambda_{(k)}
\]
sending the noncommutative symmetric functions to their symmetric function counterpart.
This leads us to the following corollary.
\begin{corollary}
\label{cor:to_sym}
  Theorem~\ref{thm:main} and Corollary~\ref{cor:main} also hold when replacing
  $\bp_r$ by the power sum symmetric function $p_r$, and $\bs_\la^{(k)}$ by the
  $k$-Schur function $s_\la^{(k)}$.
\end{corollary}

Dual $k$-Schur functions $\sig_\la^{(k)}$ indexed by $k$-bounded partitions $\la$ 
form a basis of the quotient space $\Lambda^{(k)} = \Lambda / \langle p_r \mid  r>k \rangle 
= \Lambda / \langle m_\la \mid  \la_1 > k \rangle$
(they correspond to the affine Stanley symmetric functions indexed by Grassmannian elements).
The Hall inner product $\langle \cdot, \cdot \rangle : \Lambda \times \Lambda \to \QQ$
defined by $\langle h_\la,  m_\mu \rangle = \langle s_\la, s_\mu \rangle
= \delta_{\la,\mu}$, can be restricted to $\langle \cdot, \cdot \rangle : \Lambda_{(k)}
\times \Lambda^{(k)} \to \QQ$, so that $s_\la^{(k)}$ and $\sig_\mu^{(k)}$ form dual bases
$\langle s_\la^{(k)}, \sig_\mu^{(k)} \rangle = \delta_{\la,\mu}$.
Let $z_\la$ be the size of the centralizer of any permutation of cycle type $\la$. Then
$\langle p_\la, p_\mu \rangle = z_\la \delta_{\la,\mu}$.

\begin{corollary}
\label{cor:main1}
  For $\nu$ a $k$-bounded partition, we have
  \[ \sig^{(k)}_\nu = \sum_{\lambda \in \Park{}}^{} \frac{1}{z_\lambda} \chi^{(k)}_{\nu,\lambda} \;
  p_\lambda \; . \]
\end{corollary}
\begin{proof}
Denote by $b_{\nu,\lambda}^{(k)}$ the coefficient of $p_\lambda$ in $\sig_\nu^{(k)}$, that is,
$ \sig^{(k)}_\nu = \sum_\lambda b_{\nu,\lambda}^{(k)} p_\lambda$. Then, using
Corollary~\ref{cor:main} we have 
\begin{equation*}
	z_\lambda b_{\nu,\lambda}^{(k)} = \langle p_\lambda , \sig^{(k)}_\nu \rangle
	= \langle \sum_\mu \chi^{(k)}_{\mu,\lambda} s_\mu^{(k)}, \sig_\nu^{(k)} \rangle 
	= \chi_{\nu,\lambda}^{(k)}. \qedhere
\end{equation*}
\end{proof}

Since the product of two $k$-bounded power symmetric functions is again
a $k$-bounded power symmetric function, the expansion of the dual
$k$-Schur functions in terms of $p_\la$ of Corollary~\ref{cor:main1} is
better suited for multiplication than the expansion in terms of monomial
symmetric functions. The product of two $k$-bounded monomial symmetric
functions is a sum of monomial symmetric functions which are not
necessarily $k$-bounded.

The classical Murnaghan-Nakayama rule (corresponding to sufficiently large $k$) 
has implications for representation theory.  The well-known Frobenius map sends a 
representation $V$ of the symmetric group $S_n$ to the symmetric function
\[ \sum_{\mu}^{} \frac{\chi_V(\mu)}{z_\mu} p_\mu\; , \]
where $\chi_V(\mu)$ is the character of $V$ evaluated on the conjugacy class
of type $\mu$.  This map sends the irreducible representation $V_\lambda$ to
the Schur function $s_\lambda$.  Therefore, whenever a Schur-positive
symmetric function is expanded in terms of the power-sum basis, the
coefficients can be interpreted as the character of some corresponding
representation.  However, this does not apply to Corollary~\ref{cor:main1},
since the functions $\sig^{(k)}_\lambda$ are not Schur positive.  The
possibility of a different form of a $k$-Murnaghan-Nakayama rule that would
have representation theoretical implications is discussed in
Section~\ref{s:outlook}.

The paper is organized as follows. In Section~\ref{s:notation} we introduce all notation and 
definitions. In particular, we define the various noncommutative symmetric functions.
In Section~\ref{s:MN_cores} we prove Theorem~\ref{thm:main_cores}, which is the analogue
of Theorem~\ref{thm:main} formulated in terms of the nilCoxeter algebra. 
In Section~\ref{s:equivalence} it is shown that Theorems~\ref{thm:main} and~\ref{thm:main_cores}
are equivalent. We conclude in Section~\ref{s:outlook} with some related open questions.
In Appendices~\ref{appendix:chi} and~\ref{appendix:chi-tilde} we list some tables for
$\chi_{\la,\mu}^{(k)}$ and its dual version $\tilde{\chi}_{\la,\mu}^{(k)}$.

\section*{Acknowledgments}
We thank Drexel University and the American Institute for Mathematics (AIM) for hosting two
focused research workshops in March and May 2009, where this work began.
Special thanks to Sara Billey, Luc Lapointe, Jennifer Morse, Thomas Lam, Mark Shimozono, and 
Nicolas Thi\'ery for many stimulating discussions. This work benefited greatly from
calculations run in the open source computer algebra system {\tt sage}~\cite{sage}
and {\tt sage-combinat}~\cite{sage-combinat}.

\section{Notation}
\label{s:notation}
In this section we give all necessary definitions.

\subsection{Partitions and cores}
\label{ss:partitions}
A sequence $\la=(\la_1,\la_2,\ldots,\la_\ell)$ is a partition if $\la_1\ge \la_2\ge \cdots \ge \la_\ell>0$.
We say that $\ell$ is the length of $\la$ and $|\la| = \la_1 + \cdots + \la_\ell$ is its size. 
A partition $\la$ is $k$-bounded if $\la_1\le k$. We denote by $\Park{}$ the set of all $k$-bounded
partitions.

One may represent a partition $\la$ by its partition diagram, which contains $\la_i$ boxes in
row $i$. The conjugate $\la^t$ corresponds to the diagram with rows and columns interchanged.
We use French convention and label rows in decreasing order from bottom to top. For example
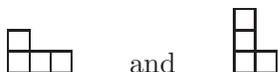
\[
  {\tiny \tableau[sbY]{|,,}} \qquad \text{and} \qquad {\tiny \tableau[sbY]{||,}}
\]
correspond to the partition $(3,1)$ and its conjugate $(2,1,1)$, respectively.

For two partitions $\la$ and $\mu$ whose diagrams are contained, that is $\la\subseteq \mu$,
we denote by $\mu/\la$ the skew partition consisting of the boxes in $\mu$ not contained
in $\la$. A ribbon is a skew shape which does not contain any $2\times 2$ squares.
An $r$-border strip is a connected ribbon with $r$ boxes.

A partition $\la$ is an $r$-core if no $r$-border strip can be removed from $\la$ 
such that the result is again a partition. For example
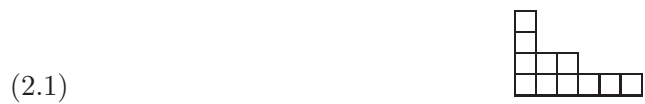
\begin{equation}
\label{e:core_example}
  {\tiny \tableau[sbY]{||,,|,,,,,|}}
\end{equation}
is a $4$-core. We denote the set of all $r$-cores by $\C_r$. 

For a cell $c=(i,j)\in \la$ in row $i$ and column $j$ we define its hook length to be the
number of cells in row $i$ of $\la$ to the right of $c$ plus the number of cells in column
$j$ of $\la$ weakly above $c$ (including $c$). An alternative definition of
an $r$-core is a partition without any cells of hook length equal to a
multiple of $r$~\cite[Ch. 1, Ex. 8]{Mac:1995}. The content of cell $c=(i,j)$
is given by $j-i \pmod{r}$.

There exists a bijection~\cite{LM:2005}
\begin{equation}
	\core_{k+1} : \Park{} \to \C_{k+1}
\end{equation}
from $k$-bounded partitions to $(k+1)$-cores defined as follows. Let $\la\in \Park{}$ considered 
as a set of cells.
Starting from the smallest row, check whether there are any cells of hook length greater than $k$.
If so, slide the row and all those in the rows 
below to the right by the minimal amount so that none of
cells in that row have a hook length greater than $k$. Then continue the procedure
with the rows below.  The positions of the cells define a skew partition
and the outer partition is a $(k+1)$-core.

The inverse map $\core_{k+1}^{-1} : \C_{k+1} \to \Park{}$ is slightly easier to
compute.  The partition $\core_{k+1}^{-1}(\kappa)$ is of the same length as the
$(k+1)$-core $\kappa$ and the $i^{th}$ entry of the partition  is the
number of cells in the the $i^{th}$ row of $\kappa$ which have a hook
smaller or equal to $k$.

Let $\la \in \Park{}$. Then the $k$-conjugate $\la^{(k)}$ of $\la$ is defined as
$\core_{k+1}^{-1}(\core_{k+1}(\la)^t )$.

\begin{example}
\label{ex:core}
For $k=3$, take $\la=(3,2,1,1)\in \Park{}$ so that
\[
  \core_4 : \quad {\tiny \tableau[sbY]{||,|,,|}}  \quad \mapsto \quad
  {\tiny \tableau[sbY]{*|*|,*,*|,,,*,*,*|}}
\]
which is the $4$-core in~\eqref{e:core_example} (where we have drawn the original
boxes of $\la$ in bold). To obtain the $k$-conjugate $\la^{(3)}$ of 
$\la$ we calculate
\[
  \core_4^{-1} : \quad {\tiny \tableau[sbY]{*|*|*|,*|,*|,,*,*|}} \qquad \mapsto \qquad 
  {\tiny \tableau[sbY]{|||||,|}} \; .
\]
\end{example}

\subsection{Affine nilCoxeter algebra}
\label{ss:affine_nilcoxeter}
The affine nilCoxeter algebra $\A_k$ is the algebra over $\ZZ$ generated by
$u_0,u_1,\ldots,u_k$ satisfying
\begin{equation}
\label{e:braid}
\begin{aligned}
	&u_i^2 = 0 && \text{for $i\in [0,k]$,}\\
	&u_i u_{i+1} u_i = u_{i+1} u_i u_{i+1} && \text{for $i\in [0,k]$,}\\
	&u_i u_j = u_j u_i && \text{for $i,j\in [0,k]$ such that $|i-j|\ge 2$,}
\end{aligned}
\end{equation}
where all indices are taken modulo $k+1$. We view the indices $i\in[0,k]$ as living on
a circle, with node $i$ being adjacent to nodes $i-1$ and $i+1$ (modulo $k+1$).
As with Coxeter groups, we have a notion of reduced words of elements $u\in \A_k$
as the shortest expressions in the generators. If $u = u_{i_1} \cdots u_{i_m}$ is a reduced
expression, we call $\{i_1,\ldots,i_m\}$ the support of $u$ denoted $\supp(u)$ (which is independent 
of the reduced word and only depends on $u$ itself). Also, $i_1\ldots i_m$ is the corresponding
reduced word and $\len(u)=m$ is the length of $u$.

A word $w$ in the letters $[0,k]$ is cyclically decreasing (resp. increasing) if the length of $w$ is 
at most $k$, every letter appears at most once, and if $i,i-1\in w$ then $i$ occurs before 
(resp. after) $i-1$. Note that since $u_i$ and $u_j$ commute if $i$ is not adjacent to $j$, all
cyclically decreasing (resp. increasing) words $w$ with the same support give rise to the same
affine nilCoxeter group element $\prod_{i\in w} u_i$.
For a proper subset $A\subsetneq [0,k]$ we define $\udecA \in \A_k$
(resp. $\uincA\in\A_k$) to be the element corresponding to cyclically decreasing (resp. increasing) words with support $A$.  
\begin{example}
Take $k=6$ and $A=\{0,2,3,4,6\}$. Then $\udecA=(u_0 u_6) (u_4 u_3 u_2) = 
(u_4 u_3 u_2) (u_0 u_6)$ and $\uincA = (u_6 u_0)(u_2 u_3 u_4) 
= (u_2 u_3 u_4)(u_6 u_0)$.
\end{example}

If $u\in \A_k$ is supported on a proper subset
$S$ of $[0,k]$, then we specify a canonical interval $I_S$ which contains the subset
$S$. Identify the smallest element $a$ (from the numbers
$0$ through $k$ with the integer order) which does not appear in
$S$. Then the canonical cyclic interval which we choose orders the elements 
$$a+1 < a+2 < \cdots < k < 0 < 1 < \cdots < a-1,$$
(where we identify $k$ and $-1$ when necessary).

\begin{definition}
\label{definition:kconnected}
An element $u\in \A_k$ (resp. word $w$) is $k$-connected if its support $S$ is an interval in
$I_S$.
\end{definition}

\begin{example}
For $k=6$, the word $w=0605$ is $k$-connected, whereas $w=06052$ is not.
\end{example}

Suppose $u\in\A_k$ has support $S\subsetneq [0,k]$. We say that $u$ corresponds to a hook
word if it has a reduced word $w$ of the form of Equation~\eqref{e:hook_word} with respect to the
canonical order $I_S$. In this case we denote by $\ascents(u)$ or $\ascents(w)$ the number of ascents 
$\ascents_{I_S}(w)$ in the canonical order. 

\begin{example}
Take $u=u_3 u_2 u_6 u_0 u_4\in \A_6$. In this case $S=\{0,2,3,4,6\}$ and 
$I_S$ is given by $2<3<4<5<6<0$. The word $w=(3)(2460)$ is a hook word with
respect to $I_S$ and $\ascents(u)=3$.
\end{example}

The generators $u_i$ in the nilCoxeter algebra $\A_k$ act on a $(k+1)$-core 
$\nu \in \C_{k+1}$ by
\begin{equation}
\label{e:nilcoxeter_action}
  u_i \cdot \nu = \begin{cases} \text{$\nu$ with all corner cells of content $i$ added if they exist,}\\
  \text{0 otherwise.} \end{cases}
\end{equation}
This action is extended to the rest of the algebra $\A_k$ and can be shown to be consistent
with the relations of the generators.
Under the bijection $\core_{k+1}^{-1}$ to $k$-bounded partitions only the topmost box added to diagram
survives. The action of $u_i$ on a $k$-bounded partition $\la$ under $\core_{k+1}$ is
denoted $u_i \cdot \la$.

\begin{example}
Taking $\nu=\core_4(\la)$ from Example~\ref{ex:core} we obtain
\[
  u_2 \cdot \nu = 
  {\tiny \tableau[sbY]{||,,,*|,,,,,,*|}}
  \qquad \text{and} \qquad
    \core_4^{-1}(u_2 \cdot \nu) = 
  {\tiny \tableau[sbY]{||,,*|,,|}}
\]
where the boxes added by $u_2$ of content 2 are indicated in bold.
\end{example}

\subsection{Noncommutative symmetric functions}
\label{ss:noncommutative}
We now give the definition of the noncommutative symmetric functions
$\be_r$, $\bh_r$, $\bs_{(r-i,1^i)}$, $\bp_r$, and $\bs_\la^{(k)}$ in terms of the affine
nilCoxeter algebra.

Following Lam~\cite{Lam:2006}, for $r=1,\ldots,k$, we define the noncommutative homogeneous 
symmetric functions 
\begin{equation*}
\bh_r = \sum_{A \in \pchoose{[0,k]}{r}} \udecA \; ,
\end{equation*}
where $\udecA$ is a cyclically decreasing element with support $A$ 
as defined in Section~\ref{ss:affine_nilcoxeter}.
We take as a defining relation for the elements $\be_r$ the equation 
$\sum_{i=0}^r (-1)^i \be_{r-i} \bh_i = 0$. It can be shown~\cite[Proposition 16]{Lam:2006} that then 
\begin{equation*}
\be_r = \sum_{A \in \pchoose{[0,k]}{r}} \uincA \; ,
\end{equation*}
where $\uincA$ is a cyclically increasing element with support $A$.
More generally, the hook Schur functions for $r\leq k$ are given by
\begin{equation*}
	\bs_{(r-i,1^i)} = \bh_{r-i} \be_i - \bh_{r-i+1} \be_{i-1} + \cdots + (-1)^i \bh_r
\end{equation*}
and we will demonstrate in Corollary \ref{cor:hookschur} (below)
that these elements may also be expressed as a sum over certain words.

The noncommutative power sum symmetric functions for $1\leq r \leq k$
are defined through the analogue
of a classical identity with ribbon Schur functions
\begin{equation*}
 	\bp_r = \sum_{i=0}^{r-1} (-1)^i \bs_{(r-i,1^i)}.
\end{equation*}

Lam~\cite[Proposition 8]{Lam:2006}  proved that, even though the variables $u_i$ do not commute, the
elements $\bh_r$ for $1\le r\le k$ commute and consequently, so do the other elements $\be_r$,
$\bp_r$, $\bs_{(r-i,1^i)}$ we have defined in terms of the $\bh_r$. We define
${\bf \Lambda}_{(k)} = \ZZ[\bh_1,\ldots,\bh_k]$ to be the noncommutative analogue of 
$\Lambda_{(k)}=\ZZ[h_1,\ldots,h_k]$.

We define the noncommutative $k$-Schur functions $\bs_\la^{(k)}$ by the noncommutative
analogue of the $k$-Pieri rule~\eqref{e:kPieri}.
Let us denote by $\HH_r^{(k)}$ the set of all pairs $(\mu,\la)$ of $k$-bounded partitions
$\mu,\la$ such that $\mu/\la$ is a horizontal $r$-strip and $\mu^{(k)}/\la^{(k)}$ is a vertical
$r$-strip (which describes the summation in the $k$-Pieri rule). Then for a $k$-bounded
partition $\la$ we require that
\begin{equation}
\label{e:kPieri_noncom}
\bh_r \bs_\la^{(k)} = \sum_{\mu: (\mu,\la) \in \HH_r^{(k)}} \bs_{\mu}^{(k)}.
\end{equation}
This definition can be used to expand the $\bh_\mu$ elements in
terms of the elements $\bs_{\la}^{(k)}$. The transition matrix is described by the
number of $k$-tableaux of given shape and weight (see~\cite{LM:2005}). 
Since this matrix is unitriangular, this system of relations can be inverted over the integers 
and hence $\{\bs_\la^{(k)} \mid \la \in \Park{} \}$ forms a basis of ${\bf \Lambda}_{(k)}$.

As shown in~\cite{LM:2005, LLMS:2006}, for $1 \leq r \leq k$, we have if $(\mu,\la) \in \HH_r^{(k)}$, then
there is a cyclically decreasing element $u \in \A_k$ of length $r$ such that $\mu = u\cdot \la$.  Moreover,
if $u\in \A_k$ is cyclically decreasing and $\mu = u \cdot \la \neq 0$, then $(\mu,\la) \in \HH_r^{(k)}$.
\begin{example}
Take $\la=(3,3,1,1) \in \mathcal{P}^{(3)}$ and $u=u_0u_3$. Then
\[
  \core_4(\la) \; = \; 
  {\tiny \tableau[sbY]{||,,,|,,,,,,|}}
  \quad \text{and} \quad
  u\cdot \core_4(\la) =
  {\tiny \tableau[sbY]{*||,*,*|,,,,*,*|,,,,,,,*,*|}}
\]
so that $((3,3,2,1,1),(3,3,1,1)) \in \HH_2^{(3)}$.
\end{example}

Hence, we may rewrite~\eqref{e:kPieri_noncom} as
\[
  \bh_r \bs_\la^{(k)} = \sum_{\mu: (\mu,\la) \in \HH_r^{(k)}} \bs_{\mu}^{(k)}
                                  =  \sum_{A \in \pchoose{[0,k]}{r} } \bs^{(k)}_{\udecA \cdot \la} \; ,
\]
where  we assume $\bs^{(k)}_{\udecA \cdot \la} = 0$ if $\udecA \cdot \la =0$.
The elements $\bh_r = \sum_{A \in \pchoose{[0,k]}{r} } \udecA$ generate 
${\bf \Lambda}_{(k)}$, and therefore more generally for any
element ${\bf f} = \sum_u c_u u \in {\bf \Lambda}_{(k)}$ with $u\in \A_k$ and $c_u\in \ZZ$
\begin{equation}
\label{e:f_on_s}
  {\bf f}~~\bs_\la^{(k)} = \sum_u c_u \bs_{u \cdot \la}^{(k)}~.
\end{equation}

Since all of the noncommutative symmetric functions in this section commute and satisfy
the same defining relations as their commutative counterparts, there is a ring isomorphism
\[
  \iota : {\bf \Lambda}_{(k)} \to \Lambda_{(k)}
\]
sending $\bh_r\mapsto h_r$, $\be_r \mapsto e_r$, $\bp_r \mapsto p_r$, $\bs_\la^{(k)} \mapsto
s_\la^{(k)}$.

\section{Main result: Murnaghan--Nakayama rule in terms of words}
\label{s:MN_cores}

We now restate Theorem~\ref{thm:main} in terms of the action of words. This result
is proved in the remainder of this section.

\begin{theorem}  \label{thm:main_cores}
For $1 \leq r \leq k$ and $\la$ a $k$-bounded partition, we have
\begin{equation} \label{eq:main_cores}
  \bp_r \bs_\lambda^{(k)} = \sum_{(w,\mu)} (-1)^{\ascents(w)} \bs_\mu^{(k)},
\end{equation}
where the sum is over all pairs $(w,\mu)$ of reduced words $w$ in the affine nilCoxeter algebra
$\A_k$ and $k$-bounded partitions $\mu$ satisfying
\begin{enumerate}
  \renewcommand{\labelenumi}{(\arabic{enumi}$'$)}
  \item (size condition) $\len(w) = r$;
  \item (ribbon condition) $w$ is a hook word;
  \item (connectedness condition) $w$ is $k$-connected;
  \item (weak order condition) $\mu= w\cdot\la$.
\end{enumerate}
\end{theorem}

In Section~\ref{s:equivalence} we will show the equivalence of Theorem~\ref{thm:main}
and Theorem~\ref{thm:main_cores}. 

The proof of Theorem~\ref{thm:main_cores} essentially amounts to computing an
expression for $\bp_r$ in terms of words.  Since all words involved will be of
length $\le k$, there will be a canonical order on the support as introduced in
Section~\ref{ss:affine_nilcoxeter}. The statistic $\ascents(w)$, and the property of being a hook word, 
will always be in terms of this canonical ordering.

\begin{lemma} \label{lem:hrmiei}For $0\le i\le r\leq k$,
  \begin{equation} \label{eq:he}
	\bh_{r-i} \be_i = \sum_w w \; ,
  \end{equation}
  where the sum is over all words $w$ satisfying $(1')$, $(2')$ with respect to the canonical
  order, and $\ascents(w)
  \in \{i-1, i\}$.
\end{lemma}
\begin{proof}
  $\bh_{r-i}$ is the sum over all cyclically decreasing nilCoxeter group elements of
  length $r-i$ and $\be_{i}$ is the sum over all cyclically increasing nilCoxeter
  group elements of length $i$.  Hence 
  $$\bh_{r-i} \be_i = \sum_{\substack{(u,v)\\\text{$u$ cycl. dec., $|u|=r-i$}\\
    \text{$v$ cycl. inc., $|v|=i$}}} uv.$$
  Rearrange each $u$ and $v$ so that they together form a hook with respect to
  the canonical order associated to the set $\supp(u) \cup \supp(v)$.  Either
  the last letter of $u$ is smaller than the first letter of $v$, in which
  case the total ascent is $i$, or the last letter of $u$ is bigger than the
  first letter in $v$, in which case the total number of ascents is $i-1$.
  This yields a bijection between hook words in the canonical order and pairs appearing 
  in this sum with the number of ascents in $\{i,i-1\}$. In the corner case $i=0$ (resp. $i=r$) the number
  of ascents
  can only be 0 (resp. $r-1$ due to the fact that the words are of length $r$).
\end{proof}

\begin{example}
  Take $k=8$, $u=(u_1u_0u_8)(u_5u_4)$ and $v=(u_2u_3)(u_0)$, so that $i=3$ and $r=8$. In this
  case the canonical order is $7<8<0<1<2<3<4<5$ and we would write $uv$ as
  $uv=[(u_5u_4)(u_1u_0u_8)][(u_0)(u_2u_3)]$, giving rise to the word 
  $w=(5410)(8023)$ with $i=3$ ascents. If on the other
  hand $u=(u_1u_0)(u_5u_4)$ and $v=(u_2u_3)(u_8u_0)$, so that $i=4$ and $r=8$, then we would
  write $uv=[(u_5u_4)(u_1u_0)][(u_8u_0)(u_2u_3)] $, giving rise to the word
  $w=(5410)(8023)$ with $i-1=3$ ascents.
\end{example}

\begin{remark}
  Note that there may be multiplicities in~\eqref{eq:he} with respect to affine
  nilCoxeter group elements because there may be several hook words with the
  same number of ascents that are equivalent to the same affine nilCoxeter element.
  For example, $(4)(20)$ and $(0)(24)$ are two different hook words with exactly one
  ascent with respect to the interval $I_{\{0,2,4\}} = \{ 2 < 4 < 0 \}$. Of
  course, they both correspond to the same affine nilCoxeter element since all
  letters in the word commute. The element with $u=u_2$ and $v=u_4u_0$
  would give rise to the hook word $w=(240)$ with $2$ ascents.
\end{remark}

We can use this lemma to get an expression for hook Schur functions.
\begin{corollary} \label{cor:hookschur}
  For $0\le i \le r \le k$, the hook Schur function is
  $$\bs_{(r-i,1^i)} = \sum_w w \; ,$$
  where the sum is over all words $w$ satisfying $(1'), (2')$ with respect to the canoncial
  order, and $\ascents(w) = i$.
\end{corollary}
\begin{proof}
  From our definition of the noncommutative Schur functions indexed by a hook partition,
  it follows that
  $$\bs_{(r-i,1^i)} = 
  \bh_{r-i} \be_i - \bh_{r-i+1} \be_{i-1} + \cdots + (-1)^i \bh_r.$$
  Hence by Lemma~\ref{lem:hrmiei} the only words which do not appear in two
  terms with opposite signs are those that have $\ascents(w) = i$, which
  implies the corollary.
\end{proof}

\begin{example}
Let $k=3$. Then for $r=3$ and $i=1$ we have
\begin{multline*}
  \bs_{2,1} = u_1 u_0 u_1 + u_2 u_1u_2 + u_3 u_2 u_3 + u_0 u_3 u_0\\
  +  u_1 u_3 u_0 + u_1 u_0 u_2 + u_2 u_0 u_1 + u_2 u_1 u_3
  + u_3 u_1 u_2  + u_3 u_2 u_0 + u_0 u_2 u_3 + u_0 u_3 u_1.
\end{multline*}
\end{example}

We can now write an expression for $\bp_r$ by using the definition.
\begin{corollary}
  For $1\le r \le k$,
    $$\bp_r = \sum_w (-1)^{\ascents(w)} w,$$
  where the sum is over all words $w$ satisfying $(1')$ and $(2')$ in the canonical order.
\end{corollary}
\begin{proof}
  This follows immediately from the definition
  $$\bp_r = \sum_{i=0}^{r-1} (-1)^i \bs_{(r-i,1^i)} \; . \qedhere$$ 
\end{proof}

In fact, we may restrict our attention to those words in the sum also satisfying
$(3')$ because it is possible to show that those not satisfying $(3')$ will cancel.
\begin{lemma}
  For $r\le k$,
  $$\bp_r = \sum_w (-1)^{\ascents(w)} w,$$
  where the sum is over all words $w$ satisfying $(1')$, $(2')$, and $(3')$.
\end{lemma}
\begin{proof}
  Since each canonical interval can be viewed as an interval of the finite
  nilCoxeter group, the sign-reversing involution described
  before~\cite[Theorem 5.1]{FG:1998} still holds and there is a sign-reversing
  involution on the terms which do not satisfy $(3')$.  Hence it suffices to
  sum only over terms which are connected cyclic intervals.
\end{proof}

\begin{example}
Let $k=3$. Then
\[
  \bp_2 = u_1 u_0 + u_2 u_1 + u_3 u_2 + u_0 u_3 - ( u_1 u_2
  + u_2 u_3 + u_3 u_0 + u_0 u_1).
\]
\end{example}

Theorem~\ref{thm:main_cores} now follows from the action of words on
$\bs_\lambda^{(k)}$ given by Equation~\eqref{e:f_on_s}.

\section{Equivalence of main theorems}
\label{s:equivalence}
To show the equivalence of Theorems~\ref{thm:main} and~\ref{thm:main_cores}, 
we will show that a $k$-bounded partition $\mu$ satisfies conditions
$(0)$ through $(4)$ of Definition~\ref{def:kribbon} if and only if there exists a
unique $w$ such that the pair $(w,\mu)$ satisfies conditions $(1')$ through
$(4')$ of Theorem~\ref{thm:main_cores}, and that such a $w$ will satisfy
$\ascents(w) = \height(\mu/\lambda)$.  

\subsection{Primed implies unprimed}

We begin by showing that conditions $(1')$ through $(4')$ of Theorem~\ref{thm:main_cores}
imply conditions $(0)$ through $(4)$ of Definition~\ref{def:kribbon}.

The first two lemmas will be important to show the correspondence between ascents in hook
words $\ascents(w)$ and the height of vertical strips $\height(\mu/\la)$,
and also for the understanding of the statistics in condition (4) of
Definition~\ref{def:kribbon}. 

\begin{lemma}
\label{lem:hook}
Let $u\in\A_k$ with $\supp(u) \subsetneq [0,k]$ and let $\mathcal{I}$ be the canonical interval
with respect to $\supp(u)$.
\begin{enumerate}
\item \label{i:hook_to_hook}
Suppose $u$ has a reduced word $\gamma_1\ldots \gamma_m \delta_1\ldots \delta_\ell$ 
such that $\gamma_1<\cdots <\gamma_m > \delta_1 > \cdots > \delta_\ell$ in $\mathcal{I}$.
Then $u$ also has a reduced word
$\beta_1 \ldots \beta_\ell \alpha_1 \ldots \alpha_m$ such that
$\beta_1 >\cdots >\beta_\ell >\alpha_1 <\cdots < \alpha_m$ in $\mathcal{I}$.
\item \label{i:unique}
If $u$ is $k$-connected and has a reduced word which is a hook word in $\mathcal{I}$,
then this hook word is unique.
\end{enumerate}
\end{lemma}

\begin{proof}
The statement~\eqref{i:hook_to_hook} follows directly from the Edelman-Greene 
insertion~\cite{EG:1987} by induction on $\ell$.  We think of the two reduced words as 
the following hook tableaux
\[
  \tableau[sbY]{\gamma_1,\ldots,\gamma_m|
                           \bl,\bl,\delta_1|
                           \bl,\bl,\vdots|
                           \bl,\bl,\delta_\ell}
 \qquad \text{and} \qquad
 \tableau[sbY]{\beta_1| \vdots| \beta_\ell | 
                          \alpha_1,\ldots,\alpha_m} \; .
\]
For $\ell=0$, the statement is trivial since $\gamma_1\ldots \gamma_m=
\alpha_1\ldots\alpha_m$ satisfies both sets of required inequalities. Suppose that the 
statement is true for $\gamma_1\ldots\gamma_m \delta_1\ldots \delta_{\ell-1}$ for $\ell>0$,
namely that this word is equivalent to $\beta_1 \ldots \beta_{\ell-1} \alpha_1'\ldots \alpha_m'$ with
$\beta_1>\cdots>\beta_{\ell-1}>\alpha_1'<\cdots <\alpha_m'$. Now insert $\delta_\ell$ in the
following way: Let $\alpha_i'$ be smallest such that $\delta_\ell<\alpha_i'$. If
$\alpha_{i-1}' \alpha_i' \delta_\ell = a (a+1) a$ for some $a\in \mathcal{I}$, then set 
$\beta_\ell=\alpha_i'$ and $\alpha_j=\alpha_j'$
for all $1\le j\le m$. Otherwise set $\beta_\ell = \alpha_i'$, $\alpha_i=\delta_\ell$ and 
$\alpha_j=\alpha_j'$ for $j\neq i$. It is not hard to see that if $\gamma_1\ldots \gamma_m 
\delta_1 \ldots \delta_{\ell-1}$ and $\beta_1\ldots \beta_{\ell-1} \alpha'_1\ldots \alpha'_m$
are equivalent, then $\gamma_1\ldots \gamma_m \delta_1 \ldots \delta_\ell$ and 
$\beta_1\ldots \beta_\ell \alpha_1\ldots \alpha_m$ are also equivalent and all 
required inequalities are satisfied. (For the inequality $\beta_\ell <
\beta_{\ell - 1}$, observe that $\beta_{\ell -1} > \delta_{\ell -1} \ge
\alpha_i' = \beta_\ell$.)

Statement~\eqref{i:unique} follows in a similar way as~\cite[Lemma 6.8]{LSS:2010}
by induction on $\len(u)$. For $\len(u)\le 3$, the uniqueness of
the hook word follows directly from the braid relations~\eqref{e:braid}.
Now let $\len(u)>3$, $w$ a hook word for $u$, and $M$ the maximal letter in $\supp(u)$.
There are two cases: either $w$ contains one or two letters $M$.

First assume that $w$ contains one $M$. Then $w=Mv$ or $vM$. Without loss of
generality we may assume that $w=Mv$ as the other case is similar. Then
$v$ is a $k$-connected hook word with $\len(v)<\len(w)$, so that by induction $v$ is unique.
By the form of the braid relations~\eqref{e:braid}, every reduced word for $u$
must have a single $M$ which precedes all $M-1$ (since $w$ does not contain any $M+1$
and the only way to obtain two $M$s is to use the braid relation $(M-1)M(M-1) \equiv
M(M-1)M$, but there is no $M-1$ to the left of $M$). Hence $w=Mv$ is the unique hook word.

Now assume that $w$ contains two $M$s, so that $w=MvM$. Suppose that $w'$ is
another hook word for $u$. Then $w'$ must contain an $M$ at the beginning or the end.
Assume without loss of generality that $w'=Mv'$. Again by induction, the hook word $v'$
is unique. Since $vM$ is also a hook word equivalent to $v'$, we must have that
$v'=vM$, which implies that $w'=w$.
\end{proof}

\begin{lemma} \label{lem:vertstrips}
  Let $\lambda \in \Park{}$ and $u \in \A_k$ with reduced word $(a-1) w a$, where $w$
  contains neither $a$ nor $a-1$ and $u \cdot \lambda = \mu \neq 0$.
  Then the cell in $\mu / \lambda$ corresponding to $a-1$ occurs
  directly above the cell in $\mu / \lambda$ corresponding to $a$.
\end{lemma}

\begin{proof}
  Recall that by~\eqref{e:nilcoxeter_action} a generator $u_i$ of $\A_k$ acts on 
  $(k+1)$-cores by adding all available boxes of residue $i$. On the $k$-bounded
  partition this amounts to adding one box (which corresponds to the topmost
  added box on the core).
  
  To show the claim of the lemma, we first show that in the core, the topmost
  added $a-1$ cannot be more than one square above the topmost added $a$.
  Suppose it were.  In that case, we consider $\core_{k+1}(\la)$ near the place
  where the topmost $a-1$ will be added.  At that location, we must have the
  following configuration:
  \[
    \tableau[sbY]{*|a}
  \]
  where the bold border represents a cell \emph{not} present in $\core_{k+1}(\la)$.

  Furthermore, we must have the following configuration at the point where
  the topmost $a$ will be added to $\core_{k+1}(\la)$:
  \[
    \tableau[mbY]{a\!-\!1,*|a,a\!+\!1} 
  \]
  (The $a$ to be added cannot be in the first column because it is below the
  diagram above).  But this means there is a removable border strip from the
  $a$ in the first diagram to the $a-1$ in the second diagram.  The length of
  this strip is a multiple of $k+1$, which is a contradiction with being a
  $(k+1)$-core.  So the topmost added $a-1$ cannot be more than one cell above
  the topmost added $a$.
  
  Now we show that the topmost added $a-1$ cannot be below the topmost added
  $a$.  Again, assume the contrary.  After adding the cells corresponding to
  $w a$, the core must be in the following configuration near the topmost added $a$:
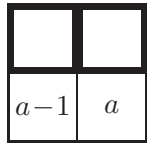
  \[
  \tableau[mbY]{*,*|a\!-\!1,a|}
  \]
  (In particular, the cell above the $a$ cannot be an addable cell, since we are
  assuming the topmost added $a-1$ will be below the topmost added $a$.)  At the same
  time, the border near where the topmost $a-1$ will be added must look like
  this:
  \[
  \tableau[mbY]{a\!-\!2,*|a\!-\!1,a}
  \]
  Thus we have a removable border strip from the $a-1$ in the top diagram to
  the $a-2$ in the bottom diagram, whose length is again a multiple of $k+1$.

  We now have that, when multiplying a core by $(a-1) w a$, the topmost added
  $a-1$ must sit directly above the topmost added $a$.  Therefore the cells
  added to the partition are necessarily in consecutive rows.  It remains to
  verify that they are in the same column.  For this, we appeal to the
  bijection $\core_{k+1}$ between $(k+1)$-cores and $k$-bounded partitions.  
  Let $y, y'$ be the rows of
  the topmost added $a-1$ and $a$, respectively.  The boxes added to the
  partition will end up in different columns if and only if the number of cells
  in $y$ with hook length greater than $k+1$ is different from the number of
  cells in $y'$ with hook length greater than $k+1$.  But every pair of
  adjacent cells with one from $y$ and one from $y'$ have hook-lengths
  differing by exactly 1.  Since no cell in a $(k+1)$-core can have hook length
  exactly $k+1$, every vertical domino in the rows $y, y'$ is either
  completely destroyed or completely preserved under the bijection,
  proving the assertion.
\end{proof}

Ultimately, we are interested in $k$-connected hook words $w$.  Such words can
be written as $w = H V$, where $H$ is a horizontal strip (strictly decreasing),
$V$ is a vertical strip (strictly increasing), and the smallest letter in $w$
is part of $V$.  The strict increase/decrease follows from the fact that we
are working in the affine nilCoxeter algebra and hence consecutive repeated letters
annihilate any partition. We can further factor $V$ into maximal segments
$v^{(i)}$ consisting of consecutive letters as
\begin{equation}
\label{e:vertical}
  V = v^{(j)} v^{(j-1)} \cdots v^{(1)} .
\end{equation}

\begin{corollary} \label{cor:vertstrips}
  Let $\la\in \Park{}$ and $V = v^{(j)} v^{(j-1)} \cdots v^{(1)}$ as described above, such that
  $V \cdot \la \neq 0$.  Then each $v^{(i)}$ adds a connected vertical 
  strip to $(v^{(i-1)}v^{(i-2)} \cdots v^{(1)}) \cdot \la$.  Futhermore, these
  strips are disjoint from one another; that is, $(V \cdot \lambda) / \lambda$
  consists of $j$ connected components.
\end{corollary}
\begin{proof}
  Let $v^{(i)} = (a-s) \cdots (a -1) a$.  Each pair $(a-r-1)(a-r)$ for $0\le
  r<s$ must correspond to a vertical domino in the skew $k$-bounded partition
  by Lemma~\ref{lem:vertstrips}.  This proves the first statement.  
  
  Now, consider any two sections of this word $v^{(\ell)}, v^{(i)}$ with $\ell >
  i$.  First note that no core-cell added by $v^{(\ell)}$ can be in any of the rows 
  containing the topmost
  vertical strip corresponding to $v^{(i)}$: this is because the lowest
  addable residue in these rows is $a+1$ which cannot appear in $v^{(\ell)}$
  by the verticality of $V$.  Thus we only have to consider the case where the
  bottom of the topmost vertical strip added by $v^{(\ell)}$ occurs in the row
  immediately above the topmost $a-s$ added by $v^{(i)}$. But notice that the
  cells of residue $a-s$ added by $v^{(i)}$ to the core have remained as a
  removable residue: no residue $a-s+1$ or $a-s-1$ has been added since.  Thus
  the corresponding partition cell must remain as a removable cell in the
  partition.  
\end{proof}

Given a skew partition $(V \cdot \lambda) / \lambda$ of the form of the previous
corollary, we call the cell directly above the vertical strip corresponding to
$v^{(i)}$ a \emph{cap} for $v^{(i)}$.  In the following lemma, we show that
if we apply a $k$-connected hook word to $\lambda$, every vertical strip
except the last will have a cap. 

\begin{lemma} \label{lem:caps}
  Let $w=H V$ be a $k$-connected hook word with $V$ as in~\eqref{e:vertical}.
  Suppose $v^{(i)} = a (a+1) \cdots (a+s)$ for a fixed $i$. If $i=j$, then
  $v^{(j)}$ does not have a cap. If $i<j$, then
  \begin{enumerate}
    \item $a-1$ occurs as a letter in $H$, and
    \item this $a-1$ forms a cap for $v^{(i)}$.
  \end{enumerate}
\end{lemma}

\begin{proof}
  First assume that $i<j$.  Since $V$ is a vertical strip and $v^{(i)}$ is of
  maximal length, the letter $(a-1)$ cannot appear anywhere in $V$. Since
  $v^{(i+1)}$ consists of letters smaller than $a$, and $w$ is $k$-connected,
  it must be the case that $(a-1)$ appears somewhere in $H$.  Note that the
  letters appearing between the $a$ in $v^{(i)}$ and the $a-1$ in $H$ are all
  strictly less than $a-1$. Hence by Lemma~\ref{lem:vertstrips}, $a-1$ must be
  a cap for $v^{(i)}$.
  
  We now show that $v^{(j)}$ does not have a cap. In this case $a$ is the
  smallest letter in $w$. Therefore, $H a$ is a horizontal strip in the
  partition, and in particular, the cell directly above the cell corresponding
  to $a$ does not appear. 
\end{proof}

We can now state precisely what happens to the height statistic when we apply
a $k$-connected hook word to a partition $\lambda$.
\begin{proposition} \label{prop:sign}
  Let $\mu,\lambda\in \Park{}$, such that $\mu = w \cdot \lambda$ for a $k$-connected
  hook word $w$ of length $r \le k$. Then
  \begin{enumerate}
  \item \label{i:asc} $\ascents(w) = \height(\mu/ \lambda)$;
  \item \label{i:stat} $\height \left(\mu / \lambda \right) +
                        \height \left(\mu^{(k)} / \lambda^{(k)} \right) = r - 1$.
  \end{enumerate}
\end{proposition}

\begin{proof}
  Factoring $w$ as $w=HV$ with the smallest letter of $w$ in $V$, we have
  $\ascents(w)=\len(V)-1$ by definition.  By Corollary~\ref{cor:vertstrips},
  we have $\height((V\cdot \la)/\la) = \len(V)-j$.  Because the cells added by $H$ form a 
  horizontal strip, a cell $c$ created by $H$ will only increase the height statistic if it forms
  the cap for some $v^{(i)}$.  By Lemma~\ref{lem:caps}, $j-1$ of the vertical strips
  will obtain a cap when applying the word $H$.  Hence 
  \[ \height(\mu/\la) = \height((V\cdot \la)/\la) + (j-1) = \len(V)-j + j -1 =
  \len(V)-1 =\ascents(w) \]
  proving~\eqref{i:asc}.
  
  To prove~\eqref{i:stat}, let $w'$ be the image of $w$ under the map which replaces every 
  letter $i$ by $k+1-i \mod (k+1)$.  It is easy to see that $w' \cdot \lambda^{(k)} = \mu^{(k)}$ if 
  $w \cdot \la = \mu$. 
  Also, if $w=HV$ as a hook word, then $w'=V'H'$ with $\len(V')=\len(H)+1$ and $\len(H')=\len(V)-1$,
  grouping the largest letter with $V'$. By Lemma~\ref{lem:hook}~\eqref{i:hook_to_hook}, the
  word $w'$ is equivalent to a hook word $w''=H''V''$ with $\len(H'')=\len(H')$ and 
  $\len(V'')=\len(V')$.  
  Hence, applying part~\eqref{i:asc} to $\lambda^{(k)}, \mu^{(k)}, w''$ we conclude that
  \[
    \ascents(w'') = \height \left( \mu^{(k)} / \lambda^{(k)} \right).
  \]
  This implies
  \begin{align*}
   \height \left(\mu / \lambda \right) + \height \left(\mu^{(k)} / \lambda^{(k)} \right) 
   = \ascents(w) + \ascents(w'') &= \len(V)-1 + \len(V'')-1\\
    &= \len(V)-1 +\len(H) = r - 1 .
    \qedhere
  \end{align*}
\end{proof}

We are finally in the position to show that the conditions of Theorem~\ref{thm:main_cores}
imply the conditions of Definition~\ref{def:kribbon}.

\begin{proposition} \label{prop:prime_implies_orig}
  If the pair $(w,\mu)$ satisfies $(1')$ through $(4')$ of Theorem~\ref{thm:main_cores}, 
  then $\mu$ must satisfy $(0)$ through $(4)$ of Theorem~\ref{thm:main}.
\end{proposition}
\begin{proof}
  By the $k$-Pieri rule, each letter in $w$ adds one cell to $\la$ and $\la^{(k)}$, which
  ensures the containment condition $(0)$.
  Since each letter in $w$ adds one box to the $k$-bounded partition $\la$, condition
  $(1)$ immediately follows from $(1')$. Condition $(3)$ is a direct translation of
  condition $(3')$ on the level of cores.
  
  To see $(2)$, note that for $\core_{k+1}(\mu)/\core_{k+1}(\la)$
  to contain a $2\times 2$ square, the word $w$ such that $\mu = w\cdot \la$ must 
  contain the pattern $a(a+1)(a-1)a$ or $a(a-1)(a+1)a$ (meaning that these have to appear as
  subwords of $w$). However, these patterns cannot appear in hook words, a contradiction.
  Hence $(2)$ follows from $(2')$.

  Condition $(4)$ follows immediately from Proposition~\ref{prop:sign}~\eqref{i:stat}.
\end{proof}

\subsection{Unprimed implies primed}
We now show that the conditions of Definition~\ref{def:kribbon} imply those of 
Theorem~\ref{thm:main_cores}. We first show that with the conditions of
Definition~\ref{def:kribbon}, there is indeed a word $w$ for an element in the affine
nilCoxeter algebra such that $\mu = w\cdot \la$.

\begin{lemma}
\label{lem:w_existence}
  Let $r \le k$, and the pair $\lambda, \mu$ satisfies conditions $(0)$ through
  $(4)$ of Definition~\ref{def:kribbon}.  Then there exists a $k$-connected hook
  word $w$ of length $r$ such that $w \cdot \lambda = \mu$.
\end{lemma}

\begin{proof}
  The proof proceeds as follows.  We first produce a $k$-connected
  hook word $w$ such that $w \cdot \lambda = \nu \supseteq \mu$.  Then we show
  that 
  \[ \height(\nu / \lambda) + \height\left( \nu^{(k)} / \lambda^{(k)}
  \right) \le \height(\mu / \lambda) + \height\left( \mu^{(k)} /
  \lambda^{(k)} \right) = r-1.\]  
  On the other hand, we can appeal to
  Proposition~\ref{prop:sign}~\eqref{i:stat} to conclude that
  \begin{equation}
  \label{e:height_nu}  
  	 \height(\nu / \lambda) + \height\left( \nu^{(k)} / \lambda^{(k)} \right) = \len(w) - 1 .
  \end{equation}
  Since $\nu \supseteq \mu$, we know that $\len(w) \ge r$.  Combining this
  with the equations above, we conclude that $\len(w) = r$.  Hence $|\nu| =
  |\mu| = |\lambda| + r$ and, since $\nu \supseteq \mu$, we must have $\nu =
  \mu$. Thus $w \cdot \lambda = \mu$.

  We now give the details of the construction of the word $w$.
  Let $\mathcal{I}$ be the set of all residues which appear in
  $\core_{k+1}(\mu) / \core_{k+1}(\lambda)$, which is an interval due to $k$-connectedness.  
  Let $\mathcal{V}$ be the set of residues occuring in the cells of 
  $\core_{k+1}(\mu) / \core_{k+1}(\lambda)$ which are
  not topmost in their column. Let $V = v_a \cdots v_0$ be the ``vertical strip
  word'' consisting of the elements of $\mathcal{V}$.  That is, $v_a <
  v_{a-1} < \cdots < v_0$ with respect to the canonical order of
  $\mathcal{I}$, and $\mathcal{V} = \{v_i\}_{i=0}^{a}$.

  We claim that $V \cdot \lambda \neq 0$ and furthermore that $\core_{k+1}(V\cdot \lambda)$
  contains all of the cells in $\core_{k+1}(\mu) / \core_{k+1}(\lambda)$ which
  are not topmost in their column.  Note that any cell which is not topmost in
  its column must be leftmost in its row, since the skew core is a ribbon.  If
  $i$ is the residue of any such cell, it will be an addable residue when the
  residues of all cells below it (in its column) have been added.  But these
  are necessarily all larger than $i$ in the canonical order.

  Now let $\mathcal{H}$ be the set of residues occuring in the cells of
  $\core_{k+1}(\mu) / \core_{k+1}(\lambda)$, which do not occur in $V\cdot \lambda$. 
  Let $H = h_b \cdots h_0$ be the ``horizontal strip word''
  consisting of the elements of $\mathcal{H}$.  That is, $h_b > h_{b-1} >
  \cdots > h_0$ with respect to the canonical order of $\mathcal{I}$, and
  $\mathcal{H} = \{h_i\}_{i=0}^b$.

  We define $w = HV$ and claim that $\nu := w \cdot \lambda \supseteq \mu$. By
  construction of $V$, any cell in $\core_{k+1}(\mu) / \core_{k+1}(\lambda)$
  which is not part of $\core_{k+1}(V \cdot \lambda)$ must be topmost in its column.
  Therefore, if $i$ is the residue of any such cell, it will be an addable
  residue when the residues of all cells to the left of it (in its row) have
  been added.  But these are necessarily all smaller than $i$ in the canonical
  order.

  Notice that the residues appearing in $w$ are precisely those occuring
  in $\mathcal{I}$, so $w$ is $k$-connected.  By construction, $w$ is also a
  hook word (although note that the smallest letter of $w$ is part of $H$, not
  $V$, according to this construction).  Since $\nu \supseteq \mu$, we must
  have $\len(w) \ge r$.  By Proposition~\ref{prop:sign}~\eqref{i:stat}, we
  conclude that~\eqref{e:height_nu} holds.

  We now claim that 
  \[
    \height(\nu / \lambda) \le \height(\mu / \lambda).
  \]
  The proof proceeds as follows.  First we note that by
  Proposition~\ref{prop:sign}~\eqref{i:asc}, we have that $\height(\nu /
  \lambda) = \ascents(w) = \len(V)$.  Next we observe that from the definition of
  the $\height$ statistic that $\height(\mu / \lambda)$ is the number of
  vertical dominos in $\mu / \lambda$ since $\mu / \lambda$ is a ribbon.  Next
  we recall that by definition, every letter $a$ in $V$ corresponds to a
  vertical domino in $\core_{k+1}(\mu) / \core_{k+1}(\lambda)$ (where $a$ is
  the residue of the bottom of the domino).  It remains to show that each of
  these dominos corresponds to a domino in $\mu / \lambda$. We first show that
  if $(a,a-1)$ occur as residues of a vertical domino in the skew core, then
  the topmost occurence of $a$ and the topmost occurence of $a-1$ occur as a
  vertical domino in the skew core.  Suppose there were an $a-1$ occuring
  above the topmost domino.  Then there would be an $a$ below it in
  $\core_{k+1}(\lambda)$, and there would also be an $a-1$ in
  $\core_{k+1}(\lambda)$ to the left of the $a$ in the domino.  This would
  form a forbidden border strip in $\core_{k+1}(\lambda)$.  Now suppose
  instead there were an $a$ occuring above the topmost domino.  This $a$ to
  the $a-1$ in the domino would form a forbidden border strip in
  $\core_{k+1}(\mu)$.  
  
  Finally, we need to show that the topmost domino in $\core_{k+1}(\mu)/ \core_{k+1}(\la)$
  with residue $(a,a-1)$ corresponds to a domino in $\mu / \lambda$. Suppose that
  these are in row $i$ and $i+1$, respectively.
  Recall that to go from a $(k+1)$-core to a $k$-bounded partition one crosses
  out all cells with hook length greater than $k+1$.
  Since $\core_{k+1}(\la)_i=\core_{k+1}(\la)_{i+1}$ and there are
  no cells with hook length $k+1$ in a $(k+1)$-core, the rightmost crossed out cell in 
  $\core_{k+1}(\la)$ under $\core_{k+1}^{-1}$ in rows $i$ and $i+1$ must be in the same column 
  $j$ and hence we must have $\la_i=\la_{i+1}$. 
  This implies that the hook length of the cell $(i+1,j+1)$ in $\core_{k+1}(\la)$
  is strictly smaller than $k$. Now look at the hook length of the
  cell $(i+1,j+1)$ in $\core_{k+1}(\mu)$. Since $\core_{k+1}(\mu)/\core_{k+1}(\la)$
  is a ribbon, there is precisely one cell in row $i+1$ in this skew shape.
  Any cells in column $j+1$ that are in $\core_{k+1}(\mu)$, but not in $\core_{k+1}(\la)$,
  must have residues different from $a$ and $a-1$ since the ones in row $i$ and $i+1$
  are topmost. This implies that the hook length
  of $(i+1,j+1)$ in $\core_{k+1}(\mu)$ is still smaller than $k+1$, so that there is a cell
  in row $i+1$ of $\mu/\la$. Since $\mu$ is a partition and $\la_i=\la_{i+1}$, there is
  also a cell in row $i$ of $\mu/\la$, so that there is a domino as desired.

  The same argument applied to the conjugate partitions shows
  \[
  \height\left(\nu^{(k)} / \lambda^{(k)}\right) \le 
  \height\left(\mu^{(k)} / \lambda^{(k)}\right).
  \]
  From this we get that $\len(w) \le r$, and we can complete the proof as
  outlined in the first paragraph.
\end{proof} 

\begin{proposition}
\label{prop:unprime_to_prime}
  Fix $\lambda \in \Park{}$.  If $\mu \in \Park{}$ satisfies conditions $(0)$ through $(4)$
  of Definition~\ref{def:kribbon}, then there exists a unique word $w$ so that the pair $(w, \mu)$ 
  satisfies the conditions $(1')$ through $(4')$ of Theorem~\ref{thm:main_cores}.
\end{proposition}

\begin{proof}
  By Lemma~\ref{lem:w_existence}, there exists a $k$-connected hook word $w$ of length $r$
  such that $\mu = w\cdot \la$. This implies the existence of the pair $(\mu,w)$ satisfying
  $(1')$ through $(4')$ of Theorem~\ref{thm:main_cores}. By 
  Lemma~\ref{lem:hook}~\eqref{i:unique} this pair is unique.
\end{proof}

By Propositions~\ref{prop:prime_implies_orig} and~\ref{prop:unprime_to_prime},
the summations in Theorems~\ref{thm:main} and~\ref{thm:main_cores} are the same.
By Proposition~\ref{prop:sign}~\eqref{i:asc} the signs also agree.
Hence Theorem~\ref{thm:main_cores} implies Theorem~\ref{thm:main}.

\section{Outlook}
\label{s:outlook}

By Corollaries~\ref{cor:main},~\ref{cor:to_sym} and~\ref{cor:main1}, the Murnaghan-Nakayama 
rule proved in this paper gives the expansion of the power sum symmetric functions in terms of 
the $k$-Schur functions $s_\la^{(k)} \in \Lambda_{(k)}$ and the expansion of the dual $k$-Schur 
functions $\sig_\la^{(k)} \in \Lambda^{(k)}$ in terms of the power sums:
\[ 
  p_\nu = \sum_{\lambda\in \Park{}}^{} \chi^{(k)}_{\lambda,\nu} \; s^{(k)}_\lambda
  \qquad \text{and} \qquad
  \sig^{(k)}_\nu = \sum_{\lambda \in \Park{}}^{} \frac{1}{z_\lambda} \chi^{(k)}_{\nu,\lambda} \;
  p_\lambda \; . 
\]
Unlike in the symmetric function case, where the Schur functions $s_\la\in \Lambda$ are self-dual,
there should be a dual version of the Murnaghan-Nakayama rule of this paper, namely 
a combinatorial formula for the coefficients $\tilde{\chi}^{(k)}_{\lambda,\nu}$  in the expansion
of the power sum symmetric functions in terms of the dual $k$-Schur functions
\[
   p_\nu = \sum_{\lambda\in \Park{}}^{} \tilde{\chi}^{(k)}_{\lambda,\nu} \; \sig^{(k)}_\lambda
\]
or, equivalently by the same arguments as in the proof of Corollary~\ref{cor:main1},
\[
  s^{(k)}_\nu = \sum_{\lambda \in \Park{}}^{} \frac{1}{z_\lambda} \tilde{\chi}^{(k)}_{\nu,\lambda} \;
  p_\lambda \; . 
\]

Since the $s^{(k)}_\nu$ are known to be Schur-positive symmetric functions~\cite{LLMS:2010}, 
they correspond to representations of the symmetric group
under the Frobenius characteristic map.  Furthermore, the characters of these
representations are given by the $\tilde{\chi}^{(k)}_{\nu,\lambda}$.  An
explicit description of such representations is an interesting open problem,
which has been studied by Li-Chung Chen and Mark Haiman~\cite{CH:2008}.  In
the most generality they conjecture a representation theoretical model for the
$k$-Schur functions with a parameter $t$ which keeps track of the degree
grading; the $\tilde{\chi}^{(k)}_{\nu,\lambda}$ described above should give
the characters of these representations without regard to degree.  Tables of
$\chi_{\la,\mu}^{(k)}$ and $\tilde{\chi}_{\la,\mu}^{(k)}$ are listed in
Appendices~\ref{appendix:chi} and~\ref{appendix:chi-tilde}.

Computer evidence suggests that the ribbon condition (2) of Definition~\ref{def:kribbon}
might be superfluous because it is implied by the other conditions of the definition. 
This was checked for $k,r\le $11 and for all $|\la|=n\le 12$ and
$|\mu| = n+r$.

\appendix
\section{Tables of $\chi_{\la,\nu}^{(k)}$}
\label{appendix:chi}

In the tables below, the partitions $\lambda$ index the row and $\nu$ indexes the column for the values of
$\chi_{\la,\nu}^{(k)}$.
\begin{center}
$k=2$, $n=3$
\begin{tabular}{|c|c|c|}
\hline
&
(111)
&
(21)\cr
\hline
(111)
&
1
&
-1
\cr\hline
(21)&
1
&
1
\cr\hline
\end{tabular}
\end{center}

\vskip .2in
\begin{center}
$k=2$, $n=4$
\begin{tabular}{|c|c|c|c|}
\hline
&
(1111)
&
(211)
&
(22)
\cr\hline
(1111)
&
1
&
-1
&
1
\cr\hline
(211)&
2
&
0
&
-2
\cr\hline
(22)&
1
&
1
&
1
\cr\hline
\end{tabular}
\end{center}

\vskip .2in
\begin{center}
$k=2$, $n=5$
\begin{tabular}{|c|c|c|c|}
\hline
&
(11111)
&
(2111)
&
(221)
\cr\hline
(11111)
&
1
&
-1
&
1
\cr\hline
(2111)&
2
&
0
&
-2
\cr\hline
(221)
&
1
&
1
&
1
\cr\hline
\end{tabular}
\end{center}

\vskip .2in
\begin{center}
$k=2$, $n=6$
\begin{tabular}{|c|c|c|c|c|}
\hline
&
(111111)
&
(21111)
&
(2211)
&
(222)
\cr\hline
(111111)
&
1
&
-1
&
1
&
-1
\cr\hline
(21111)&
3
&
-1
&
-1
&
3
\cr\hline
(2211)&
3
&
1
&
-1
&
-3
\cr\hline
(222)&
1
&
1
&
1
&
1
\cr\hline
\end{tabular}
\end{center}

\vskip .2in
\begin{center}
$k=3$, $n=4$
\begin{tabular}{|c|c|c|c|c|}
\hline
&
(1111)
&
(211)
&
(22)
&
(31)
\cr\hline
(1111)
&
1
&
-1
&
1
&
1
\cr\hline
(211)&
2
&
0
&
-2
&
-1
\cr\hline
(22)&
2
&
0
&
2
&
-1
\cr\hline
(31)&
1
&
1
&
1
&
1
\cr\hline
\end{tabular}
\end{center}

\vskip .2in
\begin{center}
$k=3$, $n=5$
\begin{tabular}{|c|c|c|c|c|c|c|}
\hline
&
(11111)&
(2111)&
(221)&
(311)
&
(32)\cr\hline
(11111)&
1
&
-1
&
1
&
1
&
-1
\cr\hline
(2111)&
3
&
-1
&
-1
&
0
&
2
\cr\hline
(221)&
4
&
0
&
0
&
-2
&
0
\cr\hline
(311)&
3
&
1
&
-1
&
0
&
-2
\cr\hline
(32)&
1
&
1
&
1
&
1
&
1
\cr\hline
\end{tabular}
\end{center}

\vskip .2in
\begin{center}
$k=3$, $n=6$
\begin{tabular}{|c|c|c|c|c|c|c|c|c|}
\hline
&
(111111)&
(21111)&
(2211)&
(3111)&
(222)&
(321)&
(33)
\cr\hline
(111111)&
1
&
-1
&
1
&
1
&
-1
&
-1
&
1
\cr\hline
(21111)&
4
&
-2
&
0
&
1
&
2
&
1
&
-2
\cr\hline
(2211)&
4
&
0
&
0
&
-2
&
-4
&
0
&
1
\cr\hline
(3111)&
6
&
0
&
-2
&
0
&
0
&
0
&
3
\cr\hline
(222)&
4
&
0
&
0
&
-2
&
4
&
0
&
1
\cr\hline
(321)&
4
&
2
&
0
&
1
&
-2
&
-1
&
-2
\cr\hline
(33)&
1
&
1
&
1
&
1
&
1
&
1
&
1
\cr\hline
\end{tabular}
\end{center}

\vskip .2in
\begin{center}
$k=4$, $n=5$
\begin{tabular}{|c|c|c|c|c|c|c|c|c|}
\hline
&
(11111)&
(2111)&
(221)&
(311)&
(32)&
(41)\cr\hline
(11111)&
1
&
-1
&
1
&
1
&
-1
&
-1
\cr\hline
(2111)&
3
&
-1
&
-1
&
0
&
2
&
1
\cr\hline
(221)&
5
&
-1
&
1
&
-1
&
-1
&
1
\cr\hline
(311)&
3
&
1
&
-1
&
0
&
-2
&
-1
\cr\hline
(32)&
5
&
1
&
1
&
-1
&
1
&
-1
\cr\hline
(41)&
1
&
1
&
1
&
1
&
1
&
1
\cr\hline
\end{tabular}
\end{center}

\vskip .2in
\begin{center}
$k=4$, $n=6$
\begin{tabular}{|c|c|c|c|c|c|c|c|c|c|c|}
\hline
&
(111111)&
(21111)&
(2211)&
(3111)&
(222)&
(321)&
(411)&
(33)&
(42)\cr\hline
(111111)&
1
&
-1
&
1
&
1
&
-1
&
-1
&
-1
&
1
&
1
\cr\hline
(21111)&
4
&
-2
&
0
&
1
&
2
&
1
&
0
&
-2
&
-2
\cr\hline
(2211)&
8
&
-2
&
0
&
-1
&
-2
&
1
&
2
&
-1
&
0
\cr\hline
(3111)&
6
&
0
&
-2
&
0
&
0
&
0
&
0
&
3
&
2
\cr\hline
(222)&
5
&
-1
&
1
&
-1
&
3
&
-1
&
1
&
2
&
-1
\cr\hline
(321)&
8
&
2
&
0
&
-1
&
2
&
-1
&
-2
&
-1
&
0
\cr\hline
(411)&
4
&
2
&
0
&
1
&
-2
&
-1
&
0
&
-2
&
-2
\cr\hline
(33)&
5
&
1
&
1
&
-1
&
-3
&
1
&
-1
&
2
&
-1
\cr\hline
(42)&
1
&
1
&
1
&
1
&
1
&
1
&
1
&
1
&
1
\cr\hline
\end{tabular}
\end{center}

\section{Tables of $\tilde{\chi}_{\la,\nu}^{(k)}$}
\label{appendix:chi-tilde}

In the tables below, the partitions $\lambda$ index the row and $\nu$ indexes the column for the values of
$\tilde{\chi}_{\la,\nu}^{(k)}$.

\begin{center}
$k=2$, $n=3$
\begin{tabular}{|c|c|c|}
\hline
&
(111)
&
(21)\cr
\hline
(111)
&
3
&
-1
\cr\hline
(21)
&
3
&
1
\cr\hline
\end{tabular}
\end{center}

\vskip .2in
\begin{center}
$k=2$, $n=4$
\begin{tabular}{|c|c|c|c|}
\hline
&
(1111)
&
(211)
&
(22)
\cr\hline
(1111)
&
6
&
-2
&
2
\cr\hline
(211)
&
6
&
0
&
-2
\cr\hline
(22)
&
6
&
2
&
2
\cr\hline
\end{tabular}
\end{center}

\vskip .2in
\begin{center}
$k=2$, $n=5$
\begin{tabular}{|c|c|c|c|}
\hline
&
(11111)
&
(2111)
&
(221)
\cr\hline
(11111)
&
30
&
-6
&
2
\cr\hline
(2111)
&
30
&
0
&
-2
\cr\hline
(221)
&
30
&
6
&
2
\cr\hline
\end{tabular}
\end{center}

\vskip .2in
\begin{center}
$k=2$, $n=6$
\begin{tabular}{|c|c|c|c|c|}
\hline
&
(111111)
&
(21111)
&
(2211)
&
(222)
\cr\hline
(111111)
&
90
&
-18
&
6
&
-6
\cr\hline
(21111)
&
90
&
-6
&
-2
&
6
\cr\hline
(2211)
&
90
&
6
&
-2
&
-6
\cr\hline
(222)
&
90
&
18
&
6
&
6
\cr\hline
\end{tabular}
\end{center}

\vskip .2in
\begin{center}
$k=3$, $n=4$
\begin{tabular}{|c|c|c|c|c|}
\hline
&
(1111)
&
(211)
&
(22)
&
(31)
\cr\hline
(1111)
&
4
&
-2
&
0
&
1
\cr\hline
(211)
&
6
&
0
&
-2
&
0
\cr\hline
(22)
&
2
&
0
&
2
&
-1
\cr\hline
(31)
&
4
&
2
&
0
&
1
\cr\hline
\end{tabular}
\end{center}

\vskip .2in
\begin{center}
$k=3$, $n=5$
\begin{tabular}{|c|c|c|c|c|c|c|}
\hline
&
(11111)&
(2111)&
(221)&
(311)
&
(32)\cr\hline
(11111)&
10
&
-4
&
2
&
1
&
-1
\cr\hline
(2111)&
10
&
-2
&
-2
&
1
&
1
\cr\hline
(221)&
10
&
0
&
2
&
-2
&
0
\cr\hline
(311)&
10
&
2
&
-2
&
1
&
-1
\cr\hline
(32)
&
10
&
4
&
2
&
1
&
1
\cr\hline
\end{tabular}
\end{center}

\vskip .2in
\begin{center}
$k=3$, $n=6$
\begin{tabular}{|c|c|c|c|c|c|c|c|c|}
\hline
&
(111111)&
(21111)&
(2211)&
(3111)&
(222)&
(321)&
(33)
\cr\hline
(111111)&
20
&
-8
&
4
&
2
&
0
&
-2
&
2
\cr\hline
(21111)&
40
&
-8
&
0
&
1
&
0
&
1
&
-2
\cr\hline
(2211)&
30
&
-2
&
2
&
-3
&
-6
&
1
&
0
\cr\hline
(3111)&
20
&
0
&
-4
&
2
&
0
&
0
&
2
\cr\hline
(222)&
30
&
2
&
2
&
-3
&
6
&
-1
&
0
\cr\hline
(321)&
40
&
8
&
0
&
1
&
0
&
-1
&
-2
\cr\hline
(33)
&
20
&
8
&
4
&
2
&
0
&
2
&
2
\cr\hline
\end{tabular}
\end{center}

\vskip .2in
\begin{center}
$k=4$, $n=5$
\begin{tabular}{|c|c|c|c|c|c|c|c|c|}
\hline
&
(11111)&
(2111)&
(221)&
(311)&
(32)&
(41)\cr\hline
(11111)&
5
&
-3
&
1
&
2
&
0
&
-1
\cr\hline
(2111)&
10
&
-2
&
-2
&
1
&
1
&
0
\cr\hline
(221)&
5
&
-1
&
1
&
-1
&
-1
&
1
\cr\hline
(311)&
10
&
2
&
-2
&
1
&
-1
&
0
\cr\hline
(32)&
5
&
1
&
1
&
-1
&
1
&
-1
\cr\hline
(41)&
5
&
3
&
1
&
2
&
0
&
1
\cr\hline
\end{tabular}
\end{center}

\vskip .2in
\begin{center}
$k=4$, $n=6$
\begin{tabular}{|c|c|c|c|c|c|c|c|c|c|c|}
\hline
&
(111111)&
(21111)&
(2211)&
(3111)&
(222)&
(321)&
(411)&
(33)&
(42)\cr\hline
(111111)&
15
&
-7
&
3
&
3
&
-3
&
-1
&
-1
&
0
&
1
\cr\hline
(21111)&
15
&
-5
&
-1
&
3
&
3
&
1
&
-1
&
0
&
-1
\cr\hline
(2211)&
25
&
-3
&
1
&
-2
&
-3
&
0
&
1
&
-2
&
1
\cr\hline
(3111)&
20
&
0
&
-4
&
2
&
0
&
0
&
0
&
2
&
0
\cr\hline
(222)&
5
&
-1
&
1
&
-1
&
3
&
-1
&
1
&
2
&
-1
\cr\hline
(321)&
25
&
3
&
1
&
-2
&
3
&
0
&
-1
&
-2
&
1
\cr\hline
(411)&
15
&
5
&
-1
&
3
&
-3
&
-1
&
1
&
0
&
-1
\cr\hline
(33)&
5
&
1
&
1
&
-1
&
-3
&
1
&
-1
&
2
&
-1
\cr\hline
(42)&
15
&
7
&
3
&
3
&
3
&
1
&
1
&
0
&
1
\cr\hline
\end{tabular}
\end{center}

\end{document}